\documentclass{amsart}
 
\usepackage{amssymb,amsmath,latexsym,amscd,amsfonts}
\usepackage{mathrsfs}
\usepackage{hyperref}
\usepackage{cleveref}
\usepackage{mathtools}

\usepackage[margin=3.85cm]{geometry}

\usepackage{graphicx}
\usepackage[all]{xy}
\usepackage{xypic}
\usepackage{psfrag}
\usepackage[svgnames]{xcolor}

\numberwithin{equation}{section}

\newtheorem{theorem}{Theorem}[section]
\newtheorem{proposition}[theorem]{Proposition}
\newtheorem{corollary}[theorem]{Corollary}
\newtheorem{lemma}[theorem]{Lemma}

\theoremstyle{definition}
\newtheorem{definition}[theorem]{Definition}

\newtheorem{remark*}[theorem]{}
\theoremstyle{remark}
\newtheorem{remark}[theorem]{Remark}

\newcommand{\rar}{\rightarrow}

\newcommand{\C}{\mathbb C}

\newcommand{\tr}{\mathrm{tr}}

\def\la{\langle}
\def\ra{\rangle}

\begin{document}
\title{On orthogonal systems, two-sided bases and regular subfactors}

\author[K C Bakshi]{Keshab Chandra Bakshi} \address{Chennai
  Mathematical Institute, Chennai, INDIA} \email{bakshi209@gmail.com, kcbakshi@cmi.ac.in}

\thanks{The first named author was supported partially by a
  postdoctoral fellowship of the National Board of Higher Mathematics
  (NBHM), India.}

\author[V P Gupta]{Ved Prakash Gupta}
\address{School of Physical Sciences, Jawaharlal Nehru University, New Delhi, INDIA}
\email{vedgupta@mail.jnu.ac.in, ved.math@gmail.com}

\maketitle
\begin{abstract}
We prove that a regular subfactor of type $II_1$ with finite Jones
index always admits a two-sided Pimsner-Popa basis. This is preceded
by a pragmatic revisit of Popa's notion of orthogonal systems.
\end{abstract}

\section{Introduction}\label{intro}
Let $\mathcal{N}\subset \mathcal{M}$ be a unital inclusion of von
Neumann algebras equipped with a faithful normal conditional
expectation $\mathcal{E}$ from $\mathcal{M}$ onto $\mathcal{N}$. Then,
a finite set $\mathcal{B}:=\{\lambda_1,\ldots,\lambda_n\}\subset
\mathcal{M}$ is called a left Pimsner-Popa basis for $\mathcal{M}$
over $\mathcal{N}$ via $\mathcal{E}$ if every $x\in \mathcal{M}$ can
be expressed as $x=\sum_{i=1}^n \mathcal{E}(x\lambda^*_i)\lambda_i$ -
see \cite{PiPo, Popa, popa2, JS, Wat} and the references therein.
Similarly, $\mathcal{B}$ is called a right Pimsner-Popa basis for
$\mathcal{M}$ over $\mathcal{N}$ via $\mathcal{E}$ if every $x\in
\mathcal{M}$ can be expressed as
$x=\sum_{j=1}^n\lambda_j\mathcal{E}(\lambda^*_jx).$ And, $\mathcal{B}$
is said to be a two-sided basis if it is simultaneously a left and a
right Pimsner-Popa basis. {It is readily seen that a
  type $II_1$ subfactor that admits a two-sided basis is always
  extremal (\Cref{2sided-extremal}).} 

An extensively exploited result of Pimsner and Popa (from \cite{PiPo})
states that if $N\subset M$ is a subfactor of type $II_1$ with finite
Jones index (\cite{Jon}), then there always exists a left
(equivalently, a right) Pimsner-Popa basis for $M$ over $N$ via the
unique trace preserving conditional expectation $E_N: M \rar N$. As
noted above, non-extremal subfactors do not admit two-sided bases.
So, it is natural to ask whether there always exists a two-sided basis
for every finite index extremal subfactor or not.  {In
  fact, it has also been asked publicly by Vaughan Jones at various places,
  see, for instance the
  talk \footnote{\url{https://www.youtube.com/watch?v=x4JN6kSie4g}
    (51:30 onward)} by M. Izumi in the workshop organized in honour of
  V.~S.~Sunder's 60th birthday in Chennai during March-April 2012. \color{black}
  Given the fact that every irreducible regular subfactor of finite
  index is a group subfactor, it is not surprising that such a
  subfactor always admits a two-sided orthonormal basis, as was
  illustrated in \cite{Hong} (also see \cite{BG}) .  However, it seems
  to be a difficult question to answer in general. In this article, we
  answer this question in affirmative for all regular subfactors of
  type $II_1$ with finite Jones index {(without assuming extremality)}
  in: \smallskip

\noindent{\bf \Cref{two-sided}.} {\em Let $N\subset M$ be a regular
  subfactor of type $II_1$ with finite Jones index. Then, $M$ admits a
  two-sided basis over $N$.}

\smallskip

{As a consequence, we deduce that every finite index
  regular subfactor of type $II_1$ is extremal.}\smallskip

Recall that an inclusion $\mathcal{Q} \subset \mathcal{P}$ of von
Neumann algebras is said to be regular if its group of normalizers
$\mathcal{N}_{\mathcal{P}}(\mathcal{Q}):=\{u\in
\mathcal{U}(\mathcal{P}):u\mathcal{Q}u^*=\mathcal{Q}\}$ generates
$\mathcal{P}$ as von Neumann algebra, i.e.,
$\mathcal{N}_{\mathcal{P}}(\mathcal{Q})'' = \mathcal{P}$. Our
proof is essentially self contained and does not depend on any
structure theorem for regular subfactors.

An effort has been made to keep this article as self-contained as
possible. The reader is assumed only to have some basic knowledge of
subfactor theory, for instance, as discussed in the first few chapters
of \cite{JS}.

Here is a brief outline of the content of this article.

As mentioned in the abstract, we first revisit, in \Cref{systems},
Popa's (\cite{Popa}) notion of an {\em orthogonal system} for an
inclusion of von Neumann algebras $\mathcal{N}\subset \mathcal{M}$
with a faithful normal conditional expectation from $\mathcal{M}$ onto
$\mathcal{N}$. This generalizes the notion of an {\em orthonormal
  basis} for a subfactor $N\subset M$ of type $II_1$ introduced by
Pimsner and Popa in \cite{PiPo}.  Dropping orthogonality, Jones and
Sunder, in \cite{JS}, generalized the notion of {\em orthonormal
  basis} and gave another formulation of basis for $M$ over $N$ (as
recalled in the first paragraph of Introduction). Very much on the
lines of \cite{JS}, we introduce and discuss the notion of a {\em
  Pimsner-Popa system}, which generalizes Popa's notion of an {\em
  orthogonal system}.

If $\mathcal{N} \subset \mathcal{M}$ is an inclusion of finite von
Neumann algebras with a fixed faithful normal tracial state $\tr$ on
$\mathcal{M}$, then for any {\em Pimsner-Popa system}
$\{\lambda_1,\cdots,\lambda_k\}$ for $\mathcal{N} \subset \mathcal{M}$
with respect to the unique $\tr$-preserving conditional expectation
from $\mathcal{M}$ onto $\mathcal{N}$, it turns out that the positive
operator $f:=\sum_{i=1}^n\lambda_ie_1{\lambda}^*_i$ is a projection in
$\mathcal{M}_1$ (\Cref{proj1}), which we call the support of the
system, where as usual $e_1$ denotes the Jones projection for the
canonical basic construction $\mathcal{N} \subset \mathcal{M} \subset
\mathcal{M}_1$. An astute reader must have already noticed that, if
the support of $\{\lambda_i\}$ equals $1$, then it is in fact a {\em
  Pimsner-Popa basis} (in the sense of \cite{JS}) for $\mathcal{M}$
over $\mathcal{N}$.

  On the other hand, for a finite index subfactor $N\subset M$ of type
  $II_1$, we observe that for every projection $f\in M_1$ there exists
  a {\em Pimsner-Popa system} with support $f$ (\Cref{basis at
    p}). An useful consequence of this observation yields: \smallskip
  
  \noindent{\bf \Cref{orthogonalsystemtobasis}} {\em Let $N \subset M$
    be a subfactor of type $II_1$ with finite index. Then, any
    Pimsner-Popa system $\{\lambda_1,\cdots,\lambda_k\}$ for $M$ over
    $N$ can be extended to a Pimsner-Popa basis for $M$ over $N$.}\smallskip

  One application being that we deduce in \Cref{popa-question} that
  every subfactor of finite index admits a Pimsner-Popa basis (not
  necessarily orthonormal) containing at least $|G|$ many unitaries,
  where $G$ is the {\it generalized Weyl group} of the subfactor (as 
  defined in the next paragraph).

Given its importance, an important example of an {\em orthogonal
  system}  for a finite index subfactor $N \subset M$
that we illustrate (in \Cref{weyl-gp}) consists of a set containing coset
representatives of, what we call, the {\em generalized Weyl group} of
the subfactor $N \subset M$,
namely, the quotient group
$$ G:=\mathcal{N}_M(N)/\mathcal{U}(N)\mathcal{U}(N^{\prime}\cap M).$$
This group was first considered by Loi in \cite{Loi}. Clearly, this
group agrees with the Weyl group of the subfactor if the subfactor is
irreducible, i.e., $N'\cap M = \C$. Such coset representatives were
also considered in \cite{Cho, JoPo, PiPo, PiPo2, Kos, Hong} in the
irreducible setup and used effectively.

Our second important class of examples of {\em
  Pimsner-Popa systems} comes from unital inclusions of finite
dimensional $C^*$-algebras - see \Cref{examples}. This is done by
employing the formalism of {\em path algebras} introduced
independently by Sunder (\cite{Sun}) and Ocneanu (\cite{Oc}). Apart
from these, \Cref{systems} is also devoted to a detailed discussion of
certain other useful properties related to {\em Pimsner-Popa systems}.

Finally, in \Cref{regular-n-2-sided-basis}, we settle the question of
existence of two-sided basis for any finite index regular subfactor $N
\subset M$.  This is achieved through a twofold strategy, namely, we
first appeal to the formalism of {\em path algebras} to get hold of a
two-sided basis for $N'\cap M$ over $\C$ with respect to the
restriction of $\tr_M$ (in \Cref{finitedimensionalbasis}), which also
turns out to be a two-sided basis for $\mathcal{R}:=N\vee
(N^{\prime}\cap M)$ over $N$ (\Cref{finitedimension1}), and then,
thanks to the regularity of $N \subset M$, every set of coset
representatives of the generalized Weyl group of $N \subset M$ turns
out to be a two-sided orthonormal basis consisting of normalizing
unitaries for $M$ over $\mathcal{R}$ (\Cref{p=ep}).  Ultimately, with
an appropriate patching technique (\Cref{two-sided-basis-normalizer}),
we deduce (in \Cref{two-sided}) that the product of these two
two-sided bases forms a two-sided Pimsner-Popa basis for $M$ over $N$.
{ And finally, employing the two-sided bases mentioned
  above and Watatani's notion of index of a conditional expectation,
  we derive (in \Cref{integer-index}) that
\[
[M:N] = |G| \mathrm{dim}_{\C}(N'\cap M),
\]
where $G$ again denotes the generalized Weyl group of the subfactor $N
\subset M$.}


\section{Pimsner-Popa bases and systems}\label{systems}
Recall, from \cite{Popa}, that given a unital inclusion of von Neumann
algebras $\mathcal{N} \subset \mathcal{M}$ with a faithful normal
conditional expectation $\mathcal{E}$ from $\mathcal{M}$ onto
$\mathcal{N}$, a family $\{m_j\}_j$ in $\mathcal{M}$ is called a right
orthogonal system for $\mathcal{M}$ over $\mathcal{N}$ with respect to
$\mathcal{E}$ if $\mathcal{E}(m_i^*m_j)= {\delta}_{ij}f_j$ for some
projections $\{f_j\}_j$ in $ \mathcal{N}$. In this article, we will be
dealing only with finite right orthogonal systems.

\subsection{Pimsner-Popa systems}
 On the lines of \cite[$\S\,
  4.3$]{JS}, Popa's notion of {\em orthogonal systems}
generalizes naturally to the following:

\begin{definition}
  Let $\mathcal{N}\subset \mathcal{M}$ be a unital inclusion of von
  Neumann algebras with a faithful normal conditional expectation
  $\mathcal{E}$ from $\mathcal{M}$ onto $\mathcal{N}$. A finite subset
  $\{\lambda_j:j\in J\}$ in $\mathcal{M}$ will be called a right
  Pimsner-Popa system for $\mathcal{M}$ over $\mathcal{N}$ with
  respect to $\mathcal{E}$ if the matrix $Q=[q_{ij}]$ with entries
  $q_{ij}:=\mathcal{E}(\lambda^*_i\lambda_j)$ is a projection in
  $M_J(\mathcal{N})$.

  Such a Pimsner-Popa system will be called a
  right orthogonal system if $q_{ij}={\delta}_{i,j}q_j$ for some
  projections $\{q_j : j \in J\} \subset \mathcal{N}$. If each $q_j$ is the
  identity operator, then such an orthogonal system will be called a
  right orthonormal system.
\end{definition}
\begin{remark}
(1) Similarly, one defines left systems by
  considering the matrix $\big[
    \mathcal{E}(\lambda_i\lambda_j^*)\big]$ in $M_J({\mathcal{N}})$.  A
  collection which is both a left system and a right system will be
  called a two-sided system. 

(2) Hereafter, by a Pimsner-Popa (resp., an orthogonal) system we will
  always mean a right Pimsner-Popa (resp., a right orthogonal) system
  and will henceforth drop the adjective `right'. And, whenever the
  conditional expectation is clear from the context, we shall omit the
  phrase `with respect to $\mathcal{E}$'.\smallskip
   \end{remark}
 In this subsection, we systematically study these objects and their
 generalities in the spirit of Pimsner-Popa basis.

 Let $\mathcal{N}\subset \mathcal{M}$ be a unital inclusion of finite
 von Neumann algebras with a fixed faithful normal tracial state $\tr$
 on $\mathcal{M}$ and let  $E_{\mathcal{N}}$ denote the unique trace preserving
normal conditional expectation from $\mathcal{M}$ onto $\mathcal{N}$.  As is
 standard, $e_1$ will denote the Jones projection that implements the
 basic construction $\mathcal{N} \subset \mathcal{M} \subset
 \mathcal{M}_1$.

\begin{lemma}\label{proj1}
 Let $\mathcal{N}\subset \mathcal{M}$, $E_{\mathcal{N}}$ be as in the preceding
 paragraph and let $\{\lambda_1,\ldots,\lambda_k\}$ be a Pimsner-Popa
 system for $\mathcal{M}/\mathcal{N}$. Then, the positive operator
 $\sum_i\lambda_ie_1\lambda^*_i$ is a projection in $\mathcal{M}_1$.
\end{lemma}
\begin{proof}
  The idea of the proof is essentially borrowed from \cite{PiPo} and
  \cite{JS}.  Consider the projection
  $Q=[q_{ij}]: = [E_{\mathcal{N}}(\lambda_i^*\lambda_j)]$ in
  $M_k(\mathcal{N})$.  Let $v_i:=\lambda_ie_1$ for $1 \leq i \leq k$
  and $V\in M_k(\mathcal{M}_1) $ be the matrix given by
  $$V=\begin{bmatrix}
         v_{1} & v_2 & \cdots & v_n\\
         0 & 0 &\cdots & 0\\
         \vdots & \vdots & \ddots & \vdots\\
         0 & 0 & \cdots & 0 
        \end{bmatrix}.$$ 
   Now, since
  $v^*_iv_j= e_1 \lambda_i^* \lambda_j e_1 = q_{ij}e_1$, we see that
  $V^*V=  Q E =E Q,
 $ where $ E$ is the diagonal matrix $\text{diag}(e_1, \ldots, e_1)$
 in $M_k(\mathcal{M}_1)$. So, $V$ is a partial isometry in
 $M_k(\mathcal{M}_1)$. In particular, $VV^*$ is a projection in
 $M_k(\mathcal{M}_1)$, thereby implying that $\sum_iv_iv^*_i =
 \sum_i\lambda_ie_1\lambda^*_i$ is a projection in $\mathcal{M}_1$.
 \end{proof}

\begin{definition}\label{support}
 Let $\mathcal{N}\subset \mathcal{M}$ and $E_{\mathcal{N}}$ be as in
 \Cref{proj1}.  For any Pimsner-Popa system $\{\lambda_i: 1 \leq i
 \leq n\}$ for $\mathcal{M}$ over $\mathcal{N}$, the projection $\sum_{i=1}^n
 \lambda_i e_1 \lambda_i^*\in \mathcal{M}_1$ will be called the
 support of the system $\{\lambda_i: 1 \leq i \leq n\}$.
\end{definition}

\begin{remark}\label{subsystem}
  \begin{enumerate}
\item A subcollection of an orthogonal (resp., orthonormal) system is
  also an orthogonal (resp., orthonormal) system.

  \item A Pimsner-Popa system with support equal to $1$ turns out to
    be a Pimsner-Popa basis for $\mathcal{M}$ over $\mathcal{N}$ (as
    mentioned in \Cref{intro}).  For such a basis, the sum
    $\sum_{i=1}^n \lambda_i\lambda^*_i$ is independent of the basis
    (see \cite{Wat}) and is called the Watatani index of $\mathcal{N}
    \subset \mathcal{M}$.  This quantity is denoted by
    $\text{Index}_w(\mathcal{N}\subset \mathcal{M})$.

  If $N\subset M$ is a finite index subfactor of type $II_1$, then it
  is known that $\text{Index}_w(N\subset M)=[M:N]$ - see \cite{Wat}
\end{enumerate}\end{remark}
The following useful equivalence is folklore and  will be used on few occasions.
\begin{lemma}\label{basis-equivalence}
 Let $\mathcal{N}\subset \mathcal{M}$ and  $E_{\mathcal{N}}$ be as in
 \Cref{proj1}. Then, for any finite set
 $\{\lambda_1,\ldots,\lambda_n\}$ in $\mathcal{M}$, 
 $\{\lambda_i:1\leq i\leq n\}$ is a  Pimsner-Popa basis for
    $\mathcal{M}/\mathcal{N}$ if and only if
 $\sum_{i=1}^n \lambda_ie_1\lambda^*_i=1$.
\end{lemma}

Unlike above characterization of a Pimsner-Popa basis
(\Cref{basis-equivalence}), the converse of \Cref{proj1} may not be
true; that is, if for some projection $f \neq 1$ in $\mathcal{M}_1$
there is a finite set $\{\lambda_i\}\subset \mathcal{M}$ satisfying
$\sum_i \lambda_ie_1{\lambda}^*_i=f$, then there is no obvious reason
why $\{\lambda_i\}$ should be a Pimsner-Popa system for
$\mathcal{M}/\mathcal{N}$. However, in some specific cases the
situation is better.
\begin{proposition}\label{proj2}
 Let $N\subset M$ be a subfactor of type $II_1$ with $[M:N]<\infty$, $\{\lambda_i: 1 \leq i \leq n\}$
 be a finite subset of $M$ and $f$ be a projection in $M_1$ satisfying
 the following three conditions:
 \begin{enumerate}
  \item $f\geq e_1$,
  \item $\sum_i\lambda_ie_1{\lambda}_i^*=f$ and
  \item $\{\lambda_i: 1 \leq i \leq n\}\subseteq \{f\}^{\prime}\cap M$.
 \end{enumerate}
Then, $\{\lambda_i:1 \leq i\leq n\}$ is a Pimsner-Popa system for $M/N$.
\end{proposition}
\begin{proof}
 Let $q_{ij} := E_N(\lambda^*_i\lambda_j)$ for $1 \leq i, j \leq
 n$. Clearly, $q_{ij}^* = q_{ji}$ and  we have
\begin{eqnarray*}
 \big(\sum_kq_{ik}q_{kj}\big)e_1
  & = & \bigg(\sum_k E_N(\lambda^*_i\lambda_k)E_N(\lambda^*_k\lambda_j)\bigg)e_1\\
  & = & \Bigg(\sum_k E_N\Big(\lambda^*_i\lambda_kE_N(\lambda^*_k \lambda_j)\Big)\Bigg)e_1\\
  & = & \sum_k e_1\lambda^*_i\lambda_kE_N(\lambda^*_k\lambda_j)e_1\\
  & = & \sum_k e_1\lambda^*_i\lambda_ke_1\lambda^*_k\lambda_je_1\\
  & = & e_1\lambda^*_if\lambda_je_1\\
  & = & e_1f\lambda^*_i\lambda_je_1\\
  & = & q_{ij}e_1
 \end{eqnarray*}
for all $1 \leq i, j\leq n$. So, by the uniqueness part of the
Pushdown Lemma \cite[Lemma 1.2]{PiPo}, we deduce that
$\sum_kq_{ik}q_{kl} = q_{ij}$ for all $1 \leq i, j \leq n$. Thus, the
matrix $Q:=[q_{ij}]$ is a projection in $M_n(N)$.  This completes the
proof.
\end{proof}
The following observation is the crux of this section.
\begin{proposition}\label{basis at p}
 Let $N\subset M$ be as in \Cref{proj2}. Then, for any projection
 $f\in M_1$, there exists a Pimsner-Popa system $\{{\lambda}_1,\ldots,
 {\lambda}_n\}$ for $ M/N$ with support equal to $f$.
\end{proposition}

\begin{proof}
The proof that we give is inspired by \cite[Proposition
  4.3.3(a)]{JS}. Fix an $n\geq [M:N]$. Since $0\leq \tr(f)\leq 1$, we
obtain $n\geq \tr(f)[M:N]$. Since $M_n(N)$ is a $II_1$-factor, we can
choose a projection $Q\in M_n(N)$ with
$\tr_{M_n(N)}(Q)=\frac{\tr(f)[M:N]}{n}$. Consider the diagonal matrix
$P_1:=\text{diag}(f, 0, \ldots, 0)$ in $M_n(M_1)$. Then, $P_1$ is a
projection with $\tr_{M_N(M_1)}(P_1)=\frac{tr(f)}{n}$.

 On the other hand, consider the projection $P_0:=QE$ in $M_n(M_1)$,
 where $E:=\text{diag}(e_1, \ldots, e_1)$.  Clearly,
 \[
\tr_{M_n(M_1)}(P_0)=\frac{\sum_i \tr(q_{ii}e_1)}{n}= \frac{\sum_i
   \tr(q_{ii})}{n\, [M:N]} = \frac{\tr_{M_n(N)}(Q)}{[M:N]}=
 \frac{\tr(f)}{n};
 \]
 so that, $P_1 \sim P_0$ in $M_n(M_1)$. Hence,
 there exists a partial isometry $V\in M_n(M_1)$ such that $V^*V=P_0$
 and $VV^*=P_1.$ Note that, the condition $VV^* =P_1$ forces $V$ to be
 of the form
 $$
 V
 = \begin{bmatrix} v_{1} & v_2 & \cdots & v_n\\ 0 & 0 &\cdots &
 0\\ \vdots & \vdots & \ddots & \vdots\\ 0 & 0 & \cdots & 0
 \end{bmatrix}$$
 for some $v_i$'s in $M_1$. These $v_i$'s then satisfy
 $\sum_iv_iv^*_i=f$ and $v^*_iv_j=q_{ij}e_1$ for all $1 \leq i, j \leq
 n$. In particular, $v^*_iv_i=q_{ii}e_1\leq e_1$ for all $1 \leq i
 \leq n$. Thus, $|v_i| \leq e_1 \leq 1$ and this implies that $|v_i| =
 |v_i| e_1$; so that, by polar decomposition of $v_i$, we obtain $v_i
 = w_i |v_i| = w_i |v_i| e_1 = v_i e_1$ for every $ 1 \leq i \leq n$,
 where each $w_i$ is an appropriate partial isometry.

Therefore, by the Pushdown Lemma \cite[Lemma
   1.2]{PiPo}, we obtain a set $\{\lambda_1,\ldots,\lambda_n\}$ in $ M$
such that $v_i=\lambda_ie_1$ for all $ 1 \leq i \leq n$. In particular,
\[
q_{ij}e_1 = v_i^*v_j = e_1 \lambda_i^* \lambda_j e_1 = E_N(\lambda_i^*
\lambda_j) e_1;
\]
so that, by the uniqueness component of Pushdown Lemma, $q_{ij} =
E_N(\lambda_i^* \lambda_j)$ for all $1 \leq i, j \leq n$. So,
$\{\lambda_1, \ldots, \lambda_n\}$ is a   Pimsner-Popa system for
$M /N$ and its support is given by $\sum_i\lambda_ie_1\lambda^*_i =
\sum_iv_iv^*_i= f.$
\end{proof}

\begin{remark}\label{ortho-basis-at-p}
  \begin{enumerate}
\item An appropriate customization of above proof actually guarantees
the existence of an orthogonal system as well.  Indeed, if we
choose a projection $q\in N$ such that $\tr(q)=\frac{\tr(f)[M:N]}{n}$
and let $Q :=\text{diag}(q,q,\ldots,q) \in M_n(N)$ then clearly $Q$ is
a projection with $\tr_{M_n(N)}(Q)=\frac{\tr(f)[M:N]}{n}$. Then, a
Pimsner-Popa system $\{\lambda_1,\cdots,\lambda_n\}$ for $M/N$
provided by the proof of Theorem \ref{basis at p} is in fact an
orthogonal system for $M/N$ with support $f$.

\item We could even take a projection $Q=(1, \ldots, 1, q)\in M_n(N)$,
  where $q$ is a projection in $N$ with $\tr_N(q) = \frac{\tr(f)[M:N] -n
    +1}{n}$. This choice of $Q$ yields an orthogonal system
  $\{\lambda_i: 1 \leq i \leq n\}$ with support $f$ such that $E_N(\lambda_i^*
  \lambda_i) = 1$ for all $ 1 \leq i \leq n-1$ and $E_N(\lambda_n^*
  \lambda_n) = q$. In particular, if $f = 1$, then we obtain an
  orthonormal basis  (in the sense of \cite{PiPo}) for $M/N$.
  \end{enumerate}
\end{remark}
As mentioned in the Introduction, the following consequence can be
used to construct bases with some specific requirements as we shall
see, for instance,  in \Cref{popa-question}.
 \begin{theorem}\label{orthogonalsystemtobasis}
  Let $N\subset M$ be as in \Cref{proj2}.  Then, any Pimsner-Popa
 system $\{\lambda_1,\ldots,\lambda_k\}$ for
  $M/N$ can be extended to a Pimsner-Popa basis
  for $M/N$.
 \end{theorem}

 \begin{proof}
Let $f$ denote the support of the given system $\{\lambda_i: 1 \leq i \leq
k\}$.  By \Cref{basis at p},   there exists a Pimsner-Popa system
$\{\lambda_{k+1},\ldots,\lambda_{k+l}\}$ for $M/N$ with support
$1-f$. Then, 
$$
\sum_{i=1}^{k+l}\lambda_ie_1\lambda^*_i=\sum_{i=1}^{k}\lambda_ie_1\lambda^*_i
+ \sum_{i=1}^{l}\lambda_{k+i}e_1\lambda^*_{k+i}= f+(1-f)=1.
$$
Thus,  by \Cref{basis-equivalence}, $\{\lambda_1,\ldots,\lambda_k,\lambda_{k+1},\ldots,\lambda_{k+l}\}$
is a Pimsner-Popa  basis for $M/N$.
 \end{proof}

 \subsection{Examples of Pimsner-Popa systems}
 \subsubsection{\bf Pimsner-Popa bases and intermediate subalgebras}\label{bases-n-intermediate}

 Let $\mathcal{N} \subset \mathcal{M}$ be an inclusion of finite von
 Neumann algebras.  Let $\mathcal{P}$ be an intermediate von Neumann
 subalgebra, i.e., $\mathcal{N}\subset \mathcal{P}\subset
 \mathcal{M}$.  Fix a faithful normal tracial state on $\mathcal{M}$
 and let $e_{\mathcal{P}}$ denote the canonical Jones projection for
 the basic construction $\mathcal{P} \subset \mathcal{M} \subset
 \mathcal{P}_1$. Let $\{{\lambda}_i\}$ be a finite set in
 $\mathcal{P}$. If $\{{\lambda}_i\}$ is a Pimsner-Popa basis for
 $\mathcal{P}/\mathcal{N}$, then it is easy to see that
 $\{\lambda_i\}$ is a Pimsner-Popa system for
 $\mathcal{M}/\mathcal{N}$ with support $e_{\mathcal{P}}$. Indeed, for
 any $x\in \mathcal{M}$, we have
 \[
 \big(\sum_i\lambda_ie_1\lambda^*_i\big)x\Omega=
 \sum_i\lambda_iE^{\mathcal{M}}_{\mathcal{N}}(\lambda^*_ix)\Omega =
 \sum_i\lambda_iE^{\mathcal{P}}_{\mathcal{N}}(\lambda^*_i
 E^{\mathcal{M}}_{\mathcal{P}}(x))\Omega =
 E^{\mathcal{M}}_{\mathcal{P}}(x)\Omega=e_{\mathcal{P}}(x\Omega),
 \]
 where the second last equality holds because $\{\lambda_i\}$ is
     a basis for $\mathcal{P}$ over $\mathcal{N}$.

 \subsubsection{\bf Inclusion of finite dimensional $C^*$-algebras}\label{examples}
  Let $A_0 \subset A_1$ be a unital inclusion of finite dimensional
  $C^*$-algebras with dimension vectors  \(
\overrightarrow{\mathit{m}}= [m_1, \cdots, m_k]
\) and \( \overrightarrow{\mathit{n}}= [n_1, \cdots, n_l], \text{ respectively};
\)
 so that 
$${A_0}\cong M_{m_1}(\C)\oplus\dots\oplus
 M_{m_k}(\C)\ \text{and}\ {A_1}\cong M_{n_1}(\C)\oplus\dots\oplus
 M_{n_l}(\C).
 $$
 We briefly recall the formalism of {\em path
   algebras} associated to such an inclusion, introduced independently
 by Ocneanu (\cite{Oc}) and Sunder (\cite{Sun}). For details, we refer
 the reader to \cite[$\S 5.4$]{JS}.

 Let $\widehat{C}$ denote the set of minimal central projections of a
 finite dimensional $C^*$-algebra $C$. With this notation, let \(
 \widehat{A_0}=\{p^{(0)}_1,\ldots, p^{(0)}_{k}\}\) and \(
 \widehat{A_1}=\{p^{(1)}_1,\ldots, p^{(1)}_l\}.  \) Let $A_{-1}:=\C$
 and put $\widehat{\C}=\{\star \}.$ Consider the Bratteli diagram for
 $\C \subset A_0$ and let $\Omega_{0]}$ denote the set of all directed
   edges starting from $\star$ and ending at $p^{(0)}_i$ for some
   $1\leq i\leq k$. Similarly, let $\Omega_{[0,1]}$ denote the set of
   edges in the Bratelli diagram of $A_0\subset A_1$, and
   $\Omega_{1]}$ denote the set of all paths starting from $\star$ and
     ending at $p^{(1)}_j$ for some $1\leq j\leq l$. For any edge or
     path $\beta$, $s(\beta)$ and $r(\beta)$ denotes the source vertex
     and range vertex of $\beta$.  Let
     $\mathcal{H}_{0]},\mathcal{H}_{[0,1]}$ and $\mathcal{H}_{1]}$
         denote the corresponding Hilbert spaces with orthonormal
         bases indexed by $\Omega_{0]},\Omega_{[0,1]}$ and
           $\Omega_{1]}$, respectively.  Then, from \cite{Sun} (also
             see \cite{JS}), there exist $C^*$-subalgebras $B_0
             \subset B_1 \subseteq \mathcal{L}(\mathcal{H}_{1]})$ such
               that the inclusion $A_0\subset A_1$ is isomorphic to
               the inclusion $B_0\subset B_1$ - see \cite[Proposition
                 5.4.1(v)]{JS}. The pair $B_0\subset B_1$ is called
               the path algebra model of the pair $A_0\subset A_1$.

      Fix $\lambda,\mu\ \in \Omega_{1]}$ with same end points. Define
  $e_{\lambda,\mu}\in B_1$ by
  \[
  e_{\lambda,\mu}(\alpha,\beta)=\delta_{\lambda,\alpha}\delta_{\mu,\beta}\text{
    for all } \alpha,\beta\in \Omega_{1]}.
    \]
Then, the set $\{e_{\lambda,\mu}:\lambda,\mu\in \Omega_{1]}\text{ with
  } r(\lambda)=r(\mu)\}$ forms a system of matrix units for $B_1$ -
  see \cite[Proposition 5.4.1 (iv)]{JS}.
  
Now, let us assume that $A_0\subset A_1$ has a faithful tracial state
$\tr$ on $A_1$.  Let $E^{A_1}_{A_{0}}: A_1 \to A_0$ denote the unique $\tr$-preserving conditional
expectation. Let $\bar{t}^{(1)}$ be the
trace vector corresponding to $\tr$ and
$\bar{t}^{(0)}$ be the one
corresponding to ${\tr}_{|_{A_0}}.$ Then, by \cite{Sun} (also see \cite{JS}), we have
\begin{equation}\label{conditionalexpectation}
E_{B_0}(e_{\lambda,\mu})=\delta_{\lambda_{[0},\mu_{[0}} \displaystyle \frac{{\bar{t}^{(1)}}_{r(\lambda)}}{{\bar{t}^{(0)}}_{r(\lambda_{0]})}} e_{\lambda_{0]},\mu_{0]}}.
\end{equation}
Now, consider $I:=\{(\kappa,\beta):\kappa \in
\Omega_{[0,1]},\beta\in \Omega_{1]},  r(\kappa)=r(\beta) \}$ and, for each  $(\kappa, \beta) \in I$, let
\[
a_{\kappa,\beta}:=\sum_{\{\theta\in\Omega_{0]}: r(\theta)=s(\kappa)\}}
  e_{\theta\circ\kappa,\beta}.
  \]
  Then, by \cite[Proposition 5.4.3]{JS}, we have 
\begin{equation}\label{formula-for-conditional-expectation}
  E_{B_0}\Big(a_{\kappa,\beta}(a_{{\kappa}^{\prime},{\beta}^{\prime}})^*\Big)=
  \delta_{(\kappa,\beta),({\kappa}^{\prime},{\beta}^{\prime})}\displaystyle
  \frac{{{\bar{t}}^{(1)}}_{r(\kappa)}}{{{\bar{t}}^{(0)}}_{s(\kappa)}}\hspace{5mm}
  \sum_{\mathclap{\substack{\theta,\theta^{\prime}\in
        \Omega_{0]}\\r(\theta)=r(\theta^{\prime})=s(\kappa)}}}
    e_{\theta,\theta^{\prime}}.
\end{equation}
Further, for each $p\in \widehat{A_0}$, consider a projection $j_p\in
B_0$ (as in \cite[Lemma 5.7.3]{JS}) given
by
\[
j_p=\frac{1}{\bar{n}^{(0)}_p}\hspace{5mm}\sum_{\mathclap{\substack{\alpha,\alpha^{\prime}\in
      \Omega_{0]}\\r(\alpha)=r(\alpha^{\prime})=p}}}
  e_{\alpha,\alpha^{\prime}},
  \]
  where $\bigg( \bar{n}_p^{(0)} \bigg)^2 = \dim p A_0$, and let
  $\lambda_{\kappa,\beta}
  :=\Bigg(\bar{n}^{(0)}_{s(\kappa)}\displaystyle
  \frac{{{\bar{t}}^{(1)}}_{r(\kappa)}}{{{\bar{t}}^{(0)}}_{s(\kappa)}}\Bigg)^{-1/2}
  a_{\kappa,\beta}.$ Then, by Equation
  \ref{formula-for-conditional-expectation}, we obtain
  \[
  E_{B_0}
  \Big(\lambda_{\kappa,\beta}(\lambda_{\kappa^{\prime},\beta^{\prime}})^*\Big)
  =
  \delta_{(\kappa,\beta),(\kappa^{\prime},\beta^{\prime})}\hspace{3mm}j_{s(\kappa)}.
  \]
  Therefore, $\{\lambda_{\kappa,\beta}:(\kappa,\beta)\in I\}$ is a left
  orthogonal system for $A_1/A_0$.  This example will have a
      significant role to play in
      \Cref{regular-n-2-sided-basis}. \smallskip
  
  We will discuss some further useful properties of Pimsner-Popa systems
in \Cref{useful-properties}. Before that, let us digress to an
important class of examples of orthonormal systems consisting of
unitaries.

\subsection{Generalized Weyl group and orthonormal systems}\label{weyl-ons}
\(\ \)

In this subsection, we illustrate an important example of an
orthonormal system consisting of unitaries, which will attract a good
share of limelight of this article. Let $N \subset M$ be a subfactor
of type $II_1$ (which is not necessarily irreducible), let
$\mathcal{U}(N)$ (resp., $\mathcal{U}(M)$) denote the group of
unitaries of $N$ (resp., $M$) and $\mathcal{N}_M(N) := \{ u \in
\mathcal{U}(M): u N u^* = N\}$ denote the group of unitary normalizers
of $N$ in $M$. It is straightforward to see that $
\mathcal{U}(N)\mathcal{U}(N^{\prime}\cap M) \big(=
\mathcal{U}(N^{\prime}\cap M)\mathcal{U}(N)\big)$ is a normal subgroup
of $\mathcal{N}_M(N)$.

\begin{definition}\cite{Loi}
  The generalized Weyl group of a subfactor $N \subset M$ is defined as the quotient
  group
  \[
  G:=\mathcal{N}_M(N)/\mathcal{U}(N)\mathcal{U}(N^{\prime}\cap
  M).
  \]
\end{definition}

This group first appeared in \cite[Proposition 5.2]{Loi}.  Note that
the generalized Weyl group of an irreducible subfactor agrees with its
Weyl group, namely, the quotient group
$\mathcal{N}_M(N)/\mathcal{U}(N)$.

 The following two useful observations are well known for irreducible
 subfactors - see, for instance, \cite{Hong, JoPo,
   PiPo,PiPo2,Kos,Loi}. For the non-irreducible case, their proofs can
 be extracted readily from \cite[Proposition 5.2]{Loi}.

\begin{lemma}\label{weyl}\cite{Loi}
Let  $w \in \mathcal{N}_M(N)\setminus \mathcal{U}(N)\,
\mathcal{U}(N^{\prime}\cap M)$. Then,
$E_N(w) = 0$.

In particular, for any two elements $v, u \in
\mathcal{N}_M(N)$, $E_N(vu^*)=0 = E_N(v^*u)$ if 
$[u]\neq [v]$ in the generalized Weyl group $G$.
\end{lemma}

\begin{corollary}\label{weyl-gp}\cite{Loi}
Suppose $[M:N]<\infty$ and $G$ denotes the generalized Weyl group of the
subfactor $N \subset M$ . Then, any set of coset representatives $\{u_g
: g =[u_g] \in G\}$ of $G$ in $\mathcal{N}_M(N)$ forms a two-sided
orthonormal system for $M/N$. Also, $G$ is a finite group with order $
\leq [M:N]$.
\end{corollary}

\begin{corollary}\label{popa-question}
Every finite index subfactor  of type $II_1$ 
admits a Pimsner-Popa basis  containing
at least $|G|$ many unitaries.
\end{corollary}
\begin{proof}
By \Cref{weyl-gp}, there exists an orthonormal system for $M/N$
consisting of unitaries. Then, by \Cref{orthogonalsystemtobasis}, this orthonormal
system can be extended to a Pimsner-Popa basis for $M/N$. This completes the proof.
  \end{proof}

\begin{remark}
  \Cref{popa-question} could be related somewhat to a recent question
  asked by Popa in \cite{popa3} about the maximum number of unitaries
  possible in an orthonormal basis (in the sense of \cite{PiPo}) of a
  given subfactor.  It, at least, tells us that every finite index
  subfactor $N \subset M$ of type $II_1$ always admits a Pimsner-Popa
  basis (not necessarily orthonormal) containing at least $|G|$ many
  unitaries. 
\end{remark}

 In view of \Cref{popa-question}, calculating cardinality of $G$
 becomes quite relevant.  However, in practice, we are yet to find a
 suitable way to calculate the cardinality of $G$. Since the
 generalized Weyl group is the same as the Weyl group of an irreducible
 subfactor, it is always non-trivial for such a subfactor.

 \subsection{Some useful properties related to  Pimsner-Popa systems}\label{useful-properties}\( \ \)

Let $(N,P,Q,M)$ be a quadruple of $II_1$-factors, i.e., $N \subset P,
Q \subset M$, with $[M:N]< \infty$. Let $\{\lambda_i:i\in I\}$ and
$\{\mu_j:j\in J\}$ be (right) Pimsner-Popa bases for $P/N$ and $Q/N$,
respectively. Consider two auxiliary operators $p(P,Q)$ and $p(Q,P)$
(as in \cite{BDLR2017}) given by
\[
p(P,Q)= \sum_{i,j}{\lambda_i}\mu_j e_1 {\mu}^*_j{\lambda}^*_i\quad \text{and}\quad
p(Q,P)= \sum_{i,j}\mu_j \lambda_i e_1 {\lambda}^*_i {\mu}^*_j.
\]
 By \cite[Lemma 2.18]{BDLR2017}, $p(P,Q)$ and $p(Q,P)$ are both
independent of choice of bases. And, by \cite[Proposition
  2.22]{BDLR2017}, $Jp(P,Q)J = p(Q,P)$, where $J$ is the usual modular
conjugation operator on $L^2(M)$; so that, $\|p(P,Q)\| =
\|p(Q,P)\|$. Let us denote this common value by $\lambda$.

\begin{proposition}
 Let $(N,P,Q,M)$ be a quadruple of type $II_1$ factors
such that $N'\cap M =\C$ and  $[M:N]<\infty$,  and let $\{\lambda_i:i\in I\}$ be a Pimsner-Popa
 basis for $P/N$. Then, the following hold:
\smallskip

 (1) $\big\{\frac{1}{\sqrt{\lambda}}\lambda_i:i\in I\big\}$ is a
 Pimsner-Popa system for $M/Q$ with support
 $\frac{1}{\lambda}p(P,Q)$.\smallskip

(2) If $(N,P,Q,M)$ is a commuting square, then $\{\lambda_i\}$
 can be extended to a Pimsner-Popa basis for $M/Q.$
\end{proposition}

\begin{proof}
(1) From \cite[Lemma 3.2]{BDLR2017}, we know that
  $\frac{1}{\lambda}p(P,Q)\, \big(=\frac{1}{\lambda}\sum_i\lambda_i
  e_Q\lambda^*_i\big)$ is a projection and, by \cite[Lemma
    3.4]{BDLR2017}, $e_Q$ is a subprojection of
  $\frac{1}{\lambda}p(P,Q)$. Further, by \cite[Proposition
    2.25]{BDLR2017}, we know that $p(P,Q)\in P^{\prime}\cap Q_1$; so,
  it follows that $\big\{\lambda_i:i\in I\big\}\subseteq
  \big\{\frac{1}{\lambda}p(P,Q)\big\}^{\prime}\cap M$.  Also, we have
  $\sum_i \frac{1}{\sqrt{\lambda}} \lambda_i e_Q
  \frac{1}{\sqrt{\lambda}} \lambda_i^* = \frac{1}{\lambda}p(P,Q)$.
  Thus, in view of \Cref{proj2}, $\big\{\frac{1}{\sqrt{\lambda}}
  \lambda_i:i\in I\big\}$ is a Pimsner-Popa system for $M/Q$ with
  support $\frac{1}{\lambda}p(P,Q)$\smallskip

(2) Suppose that $(N,P,Q,M)$ is a commuting square. Then, by
 \cite[Propositions 2.14 $\&$ 2.20]{BDLR2017}, we know that $p(P,Q)$
 is a projection. Thus, $\lambda = \|p(P,Q)\| = 1$ and the conclusion
 follows from (1) and \Cref{orthogonalsystemtobasis}.
\end{proof}

\begin{proposition}
 Let $N\subset M$ be an irreducible subfactor of type $II_1$ with
 finite index and $\{\lambda_i\}$ be a Pimsner-Popa system for $M/N$
 with support lying in $N^{\prime}\cap M_1$. Then, $1\leq
 \sum_i\lambda_i{\lambda_i}^*\leq [M:N].$
\end{proposition}
\begin{proof}
 Let $f$ denote the support of $\{\lambda_i\}$, i.e., $ f
 =\sum_i\lambda_ie_1{\lambda}^*_i$. Then, we obtain
 $\sum_i\lambda_i{\lambda}^*_i= [M:N]E_M(f).$ Since $N^{\prime}\cap
 M=\C$, we have $E_M(f)=\tr(f)\in [0,1].$ Therefore,
 $\sum_i\lambda_i{\lambda}^*_i\leq [M:N]$.

 On the other hand, since $f \in N'\cap M_1$ and $N'\cap M= \C$, by
 \cite[Proposition 1.9]{PiPo}, we have $\tr(f) \geq \tau$. Then, by
 irreducibility of $N \subset M$ again, we have $\tr(f) = E_M(f) = \tau
 \sum_i\lambda_i{\lambda}^*_i$. Hence, $ \sum_i\lambda_i{\lambda}^*_i
 \geq 1$.
 \end{proof}

We conclude this section with a small observation on a kind of local
behaviour of orthogonal systems. Recall, from \cite{Jon}, that for a
subfactor $N\subset M$ and a projection $f\in N^{\prime}\cap M$, the
index of $N$ at $f$ is given by $[M_f:N_f]= {[M:N]}_f$.  Also, a finite
index subfactor $N \subset M$ is said to be extremal, if $\tr_{N'}$
and $\tr_{M}$ agree on $N'\cap M$. Clearly, if $N\subset M$ is
irreducible, then it is extremal.\color{black}

\begin{proposition}
 Let $N\subset M$ be an irreducible subfactor of type $II_1$ with
 $[M:N]<\infty$ and $f\in N^{\prime}\cap M_1$ be a projection. Then,
 for any orthogonal system $\{\lambda_i\}$ with support $f$, we have
 $\sum_i\lambda_i{\lambda}^*_i=\sqrt{{[M_1:N]}_f}.$
\end{proposition}
\begin{proof}
 Since $N\subset M$ is extremal, the following local index formula
 holds (see \cite{Jon}):
 \[
 [fM_1f:Nf]=[M_1:N]{\big(\tr_{M_1}(f)\big)}^2=\Big([M:N]\tr_{M_1}(f)\Big)^2.
 \]
 On the other hand, since $\{\lambda_i\}$ is an orthogonal system, we
 obtain $\sum_i \lambda_i\lambda^*_i=[M:N]\tr_{M_1}(f).$ This completes the
 proof.
\end{proof}

\section{Regular subfactor and two-sided basis}\label{regular-n-2-sided-basis}
Before we pursue our hunt for a two-sided basis in a
regular subfactor, as asserted in the Introduction, we first show that
every finite index subfactor with a two-sided basis is extremal,
which, most likely, is folklore.

\begin{proposition}\label{2sided-extremal}
Let $N \subset M$ be a type $II_1$ subfactor with finite index. If
there exists a two-sided basis for $M$ over $N$, then $N \subset M$ is
extremal.
\end{proposition}
\begin{proof}
Given any right basis $\{\lambda_i : 1 \leq i \leq n\}$ for $M/N$, it is
known (see, for instance, \cite[Lemma 2.23]{BDLR2017}) that the $\tr_{N'}$
preserving conditional expectation $E_{M'} : N' \rar M'$ is given by
\[
E_{M'}(x) = [M:N]^{-1}\sum_i \lambda_i x \lambda_i^*,\  x \in N'.
\]
Thus, if $x \in N'\cap M$, then
\(
\tr_{N'}(x) = E_{M'\cap M}(x) = [M:N]^{-1}\sum_i \lambda_i x \lambda_i^*.
\)

Now, let $\{\lambda_i : 1 \leq i \leq n\}$ be any two-sided basis for
$M/N$. Then, we have $\sum_i \lambda_i^* e_1 \lambda_i = 1 = \sum_i
\lambda_i e_1 \lambda_i^*$ so that $\sum_i\lambda_i^* \lambda_i =
       [M:N]\, 1_M$ (after applying $E^{M_1}_M$ on both sides of first
       equality). Thus, for any $x \in N'\cap M$, we have
\begin{eqnarray*}
  \tr_{M}(x) & = & [M:N]^{-1} \tr_{M}\big(x \sum_i \lambda_i^* \lambda_i\big) \\
  & = &
  [M:N]^{-1} \tr_{M}\big( \sum_i \lambda_i x \lambda_i^*\big)\\ &  = & \tr_M\big( \tr_{N'}(x)\, 1_M \big)\\
  &= & \tr_{N'}(x).
       \end{eqnarray*}
       Hence, $N \subset M$ is extremal.  \end{proof}

 As the header suggests, this section is devoted to
proving the existence of two-sided basis for a finite index regular
subfactor. Keeping this in mind, from now onward, throughout this
section, $N\subset M$ will denote a finite index subfactor of type
$II_1$, which is not necessarily irreducible, and $\mathcal{R}$ will
denote the intermediate von Neumann subalgebra generated by $N $ and
$N '\cap M$, i.e., $\mathcal{R} = N\vee (N^{\prime}\cap M$). We first
present some preparatory results that we require to deduce the main
theorem.

\begin{lemma}\label{samenormalizer}
  With notations as in the preceding paragraph, we have
  \[
  \mathcal{N}_M(N)\subseteq \mathcal{N}_M\big(\mathcal{R}\big).
  \]
\end{lemma}
\begin{proof}
Let $u\in \mathcal{N}_M(N)$. Then, $uNu^*=N$, and for $x\in
N^{\prime}\cap M$, we have
\[
(uxu^*)n=uxu^*nuu^*=uu^*nuxu^*=n(uxu^*)\ \text{for all } n \in N,
\] i.e., $u(N^{\prime}\cap
M)u^*=N^{\prime}\cap M$. So, $u(nx)u^* = (unu^*)(uxu^*)\in
N\vee(N'\cap M)$ for all $n \in N$ and $x \in N'\cap M$. Thus, we readily deduce that $u\mathcal{R}u^* = \mathcal{R}$.
\end{proof}

The following crucial ingredient is an adaptation of \cite[Lemma 5.7.3]{JS}.
\begin{proposition}\label{finitedimensionalbasis}
Let $\tr$ denote the restriction of $\tr_M$ on $N^{\prime}\cap M$. Then,
$N^{\prime}\cap M$ has a two-sided Pimsner-Popa basis over $\C$ with
respect to $\tr$. 
\end{proposition}
\begin{proof}
Let $\overrightarrow{n}=[n_1,n_2,\cdots,n_k]$ denote the dimension
vector of $N^{\prime}\cap M$ and $\bar{t}$ denote the trace vector of
$\tr$.  Consider the path algebra model $B_{-1} \subseteq B_0 \subseteq
B_1$ for the inclusion $\C \subseteq N'\cap M$ as recalled in
\Cref{examples}. Since $(\C \subseteq N'\cap M) \cong (B_0 \subseteq B_1)$,
it is enough to show that $B_0 \subseteq B_1$ admits a two-sided basis
with respect to the tracial state (on $B_1$) determined by the trace
vector $\bar{t}$. Let
\[
J:=\{(\kappa,\beta):\kappa,\beta\in
\Omega_{1]}~~~\text{such that}~~r(\kappa)=r(\beta)\}.
  \]
  Then, by \cite[Proposition 5.4.1(iv)]{JS} (or see \Cref{examples}),
    $\{e_{\kappa,\beta} : (\kappa,\beta) \in J\}$ is a system of
    matrix units for $B_1$. So, by \cite[Proposition 5.4.3 (iii)]{JS},
    we easily deduce that
  \[
  E_{B_0}\big(e_{\kappa,\beta}(e_{\kappa^{\prime},\beta^{\prime}})^*\big) = \delta_{(\kappa,\beta),(\kappa^{\prime},\beta^{\prime})}\bar{t}_{r(\kappa)}\text{ 
    for all } (\kappa,\beta),(\kappa^{\prime},\beta^{\prime}) \in J.
  \]
Then, defining
\[
\lambda_{\kappa,\beta}=\frac{1}{\sqrt{\, \bar{t}_{r(\kappa)}}}e_{\kappa,\beta}\text{ 
  for }(\kappa,\beta) \in J,
\]
we obtain
\[
\sum_{(\kappa^{\prime},\beta^{\prime})\in
  I}E_{B_0}\big(e_{\kappa,\beta}(\lambda_{\kappa^{\prime},\beta^{\prime}})^*\big)\lambda_{\kappa^{\prime},\beta^{\prime}}=e_{\kappa,\beta}\ \text{
  for all } (\kappa,\beta) \in J.
\]
In particular, since $\{e_{\kappa,\beta} : (\kappa,\beta) \in J\}$ is a system of
matrix units for $B_1$, we  have 
\[
\sum_{(\kappa^{\prime},\beta^{\prime})\in
  J}E_{B_0}\big(x(\lambda_{\kappa^{\prime},\beta^{\prime}})^*\big)\lambda_{\kappa^{\prime},\beta^{\prime}}=x\ \text{
  for all } x \in B_1,
\]
that is, $\mathcal{B}:=\{
\lambda_{\kappa^{\prime},\beta^{\prime}}:(\kappa^{\prime},\beta^{\prime})\in
J \}$ is a left Pimsner-Popa basis for $(N^{\prime}\cap M)/\C$. Hence, being a
self-adjoint set, $\mathcal{B}$ is in fact a two-sided Pimsner-Popa
basis for $B_1$ over $\C$.
\end{proof}

\begin{lemma}\label{finitedimension1}
$\mathcal{R}$ has a two-sided basis over $N$  contained in $
  N^{\prime}\cap M.$
\end{lemma}
\begin{proof}
 First observe that $(\mathbb{C},N^{\prime}\cap M,N, M)$ is a
 commuting square (see, for instance, \cite[ Lemma 4.6.2]{GHJ}).  Now
 the quadruple $\big(\mathbb{C},N^{\prime}\cap M,N,N\vee
 (N^{\prime}\cap M)\big)$ is non-degenerate because $N\vee
 (N^{\prime}\cap M)$ is the SOT closure of the algebra
 $N(N^{\prime}\cap M)$ $(= (N^{\prime}\cap M)N)$ (see \cite[Proposition
   1.1.5]{popa2}).  Therefore, the conclusion follows once we apply
 Lemma \ref{finitedimensionalbasis} and \cite[Proposition
   1.1.5]{popa2} again.
\end{proof}

The following useful result is implicit in \cite{Kall}, and was also
observed in \cite[Lemma 4.2]{Cam}. For the sake of completeness, we
include a proof using Pimsner-Popa basis.

\begin{lemma}\label{kallman}
 Let $\theta$ be an automorphism of $\mathcal{R}$ such that its
 restriction to $N$ is an outer automorphism of
 $N$.  Then, $\theta$ is a free automorphism of $\mathcal{R}$.
\end{lemma}
\begin{proof}
 Suppose $\theta$ is not a free automorphism of $\mathcal{R}$. Then,
 by definition, there exists a non-zero $r\in \mathcal{R}$ such that
 \begin{equation}\label{free}
 rx=\theta(x)r\ \text{for all } x \in \mathcal{R}.
  \end{equation}
By Lemma \ref{finitedimension1}, there exists a basis
$\{\lambda_1,\ldots,\lambda_n\}$ for $\mathcal{R}/N$ contained in
$N^{\prime}\cap M$. Since $\sum_{i=1}^k
\lambda_iE_N(\lambda^*_ir)=r\neq 0$, we must have $E_N(\lambda^*_jr)\neq 0$
for at least one $\lambda_j$. Thus,
 multiplying both sides of \Cref{free} by $\lambda^*_j$ on the left,  we obtain
\begin{equation}\label{free2}
 \lambda^*_jrx=\lambda^*_j\theta(x)r =\theta(x)\lambda^*_jr \ \text{
   for all }x\in N.
\end{equation}
Then, taking conditional expectation $E_N$ on both sideds of
\Cref{free2}, we get
$$
E_N(\lambda^*_jr)x=\theta(x)E_N(\lambda^*_jr)\ \text{ for all } x \in N.
$$ This shows that $\theta|_{N}$ is not free. But a free automorphism
of a factor is outer (\cite{Kall}, \cite[$\S A.4$]{JS}). Hence, we have a
contradiction as $\theta_{|_N}$ is given to be outer.
\end{proof}

\begin{proposition}\label{orthonormalsystem-of-Q}
Let $G$ denote the generalized Weyl group of $N \subset M$. Then, any
set of coset representatives $\{u_g : g =[u_g] \in G\}$ of $G$ in
$\mathcal{N}_M(N)$ forms a two-sided orthonormal system for
$M/\mathcal{R}$.
\end{proposition}

\begin{proof}
  Let $w\in \mathcal{N}_M(N)$. We first assert that
  \[
  E_{\mathcal{R} }(w)=0 \text{ if and only if } w \in
  \mathcal{N}_M(N)\setminus \mathcal{U}(N) \mathcal{U}(N'\cap M).
  \]
Necessity is obvious. Conversely, suppose $w \notin
\mathcal{U}(N)\mathcal{U}(N'\cap M) $. Note that, by Lemma
\ref{samenormalizer}, $wxw^*\in \mathcal{R}$ for all $x\in
\mathcal{R}$. So, $\beta: \mathcal{R} \rightarrow \mathcal{R} $
defined by $\beta(x)=wxw^*$ is an automorphism of $ \mathcal{R}$,
which restricts to an outer automorphism on $N$ (since $w \notin
\mathcal{U}(N)\mathcal{U}(N'\cap M) $). Thus, by Proposition
\ref{kallman}, $\beta$ is a free automorphism of $\mathcal{R}$. Then,
applying $E_{ \mathcal{R}}$ on both sides of the equation
$wx=\beta(x)w$, we obtain $E_{\mathcal{R} }(w)x=\beta(x)E_N(w)$ for
all $x \in \mathcal{R} $. Since $\beta$ is free, we must have
$E_{\mathcal{R} }(w)=0$. This proves the assertion.

Now, fix a set of coset
representatives $\{u_g : g =[u_g] \in G\}$ of $G$ in
$\mathcal{N}_M(N)$. Then, by above assertion, we have
  \begin{equation}\label{R-orthogonality}
  E_{\mathcal{R}
  }(u_gu^*_h)=0=E_{ \mathcal{R} }(u_g^*u_h) ~~\text{if and only if}~~~ g\neq h.
  \end{equation}
  Hence, $\{u_g : g \in G\}$ forms a two-sided orthonormal system for $M$
  over $ \mathcal{R} $.
\end{proof}

\begin{proposition}\label{p=ep}
Let $\mathcal{P}:={\mathcal{N}_M(N)}''$ and $\{u_g : g \in G\}$ be an
orthonormal system for $M/\mathcal{R}$ as in
\Cref{orthonormalsystem-of-Q}. If $p$ denotes the support of $\{u_g :
g \in G\}$, then $p=e_{\mathcal{P}}$.

In particular, if $N \subset M$ is regular, then $\{u_g:g\in G\}$
forms a two-sided orthonormal basis for $M$ over $\mathcal{R}$.
\end{proposition}
\begin{proof}
  We have $p=\sum_g u_ge_{\mathcal{R}}u^*_g\in
  \la \mathcal{M}, e_{\mathcal{R}}\ra \in B(L^2(\mathcal{M}))$ (see
  \Cref{support}). We first assert that $p|_{L^2(\mathcal{P})}=id$.

Let $A = \text{span}\big(\mathcal{N}_M(N)\big).$ Then, $\mathcal{P} =
A''$ and since $A$ is a unital $*$-subalgebra of $\mathcal{P}$, by
Double Commutant Theorem, we have $A'' = \overline{A}^{\mathrm{SOT}}$. Let $x
\in \mathcal{P}.$ Then, there exists a net $(x_i) \subset A$ such that
$x_i$ converges to $x$ in SOT. Thus, $x_i\Omega$ converges to
$x\Omega$ in $L^2(M)$. So, it suffices to show that $p(u\Omega) = u
\Omega$ for every $u \in \mathcal{N}_M(N)$ for then we will have
  \[
  p(x \Omega) = \lim_i p(x_i \Omega) =  \lim_i x_i \Omega  =  x\Omega.
  \]
  Let $u \in \mathcal{N}_M(N)$.  Then, $[u] = [u_g]$ for a unique $g
  \in G$. So, $u = u_g v$ for some $v \in
  \mathcal{U}(N)\mathcal{U}(N'\cap M)$. Thus,
\[
  p(u \Omega)  =  \sum_{t \in G} u_t e_{\mathcal{R}} u_t^* u \Omega
=  \sum_{t \in G} u_t E_\mathcal{R}( u_t^* u) \Omega
  =  \sum_{t \in G} u_t E_\mathcal{R}( u_t^* u_g)v \Omega
     =  u_g v \Omega
    =  u \Omega,
   \]
   where the second last equality holds because of Equation
   \ref{R-orthogonality}.

Now, it just remains to  show that $p_{|_{\big({L^2(P)}\big)^{\perp}}}=0.$ For this, it
suffices to show that, for all $y\in M$ satisfying $\tr_M(x^*y)=0$ for all
$x\in \mathcal{P}$, we must have $p(y\Omega)=0,$ that is, we just need to show
that $\sum_{g \in G} u_gE_{\mathcal{R}}(u^*_gy)\Omega=0$ for any such $y$. In
fact, we assert that $E_{\mathcal{R}}(u^*_gy)=0$ for all $g \in G$.

For $z\in \mathcal{U}(N)\mathcal{U}(N^{\prime}\cap M)$, $u_gz^*\in
\mathcal{P}$ so that $\tr_M(zu^*_gy)=0$ for all $g \in G$.  Further, since
$ \mathcal{R} =
\overline{\text{span}{\{\mathcal{U}(N)\mathcal{U}(N^{\prime}\cap
    M)\}}}^{\mathrm{SOT}}$ and  $\tr_M$ is SOT-continuous on bounded
  sets, it follows that $ \tr_M(ru^*_gy)=0$ for all $r\in \mathcal{R}$
and $g \in G$.  Hence, by the trace
preserving property of the conditional expectation, we deduce that
$E_{\mathcal{R}}(u^*_gy)=0$ for all $g \in G$. This completes the proof.
\end{proof}

 The following two elementary observations turn out to be catalytic in
 proving the existence of two-sided basis for an arbitrary regular
 subfactor of type $II_1$ with finite index.

 \begin{lemma}\label{normalizer-and-basis}
  Let $\mathcal{N}\subset \mathcal{P}\subset \mathcal{M}$ be an
  inclusion of finite von Neumann algebras  with a faithful tracial
  state $\tr$ on $\mathcal{M}$ and $\{\lambda_i:1\leq i\leq m\}$ be a
  basis for $\mathcal{P}/\mathcal{N}$.  Then, for any $u\in
  \mathcal{N}_{\mathcal{M}}(\mathcal{P})\cap
  \mathcal{N}_{\mathcal{M}}(\mathcal{N})$, $\{u\lambda_iu^*:1\leq
  i\leq m\}$ is also a basis for $\mathcal{P}/\mathcal{N}$.
 \end{lemma}
 \begin{proof}
Note that the map $\theta: \mathcal{P}\rar \mathcal{P}$ given by
$\theta(x) = u x u^*$ is a $\tr_{\mathcal{M}}$ (and hence
$\tr_{\mathcal{P}}$) preserving automorphism of $\mathcal{P}$ that
keeps $\mathcal{N}$ invariant. Then, a routine verification shows that
$\{u\lambda_iu^*:1\leq i\leq m\}$ is also a basis for
$\mathcal{P}/\mathcal{N}$, which we leave  to the reader.
\end{proof}

 \begin{proposition}\label{two-sided-basis-normalizer}
  Let $\mathcal{N}\subset \mathcal{P}\subset \mathcal{M}$ be as in
  \Cref{normalizer-and-basis}. Suppose $\mathcal{P}/\mathcal{N}$ has a
  two-sided basis $\{\lambda_i: 1\leq i\leq m\}$ and
  $\mathcal{M}/\mathcal{P}$ has a two-sided basis $\{\mu_j:1\leq j\leq
  n\}$ contained in $ \mathcal{N}_{\mathcal{M}}(\mathcal{P})\cap
  \mathcal{N}_{\mathcal{M}}(\mathcal{N}).$ Then, $\{\mu_j\lambda_i:
  1\leq i\leq m,   1\leq j\leq n\}$ is a two-sided basis
  for $\mathcal{M}/\mathcal{N}$.
 \end{proposition}
\begin{proof}
 Let $ \lambda^{\prime}_{i,j}:=\mu_j\lambda_i\mu^*_j$, $1 \leq i \leq
 m, 1 \leq j \leq n$. Then, by Lemma \ref{normalizer-and-basis},
 $\{\lambda^{\prime}_{i,j}:1\leq i\leq m\}$ is a basis for
 $\mathcal{P}/\mathcal{N}$ for each $j$. Similarly,
 $\{(\lambda^{\prime}_{i,j})^*:1\leq i\leq m\}$ is also a basis for
 $\mathcal{P}/\mathcal{N}$. Since $\{\lambda_i\}$ is a basis for
 $\mathcal{P}/\mathcal{N}$, we have
 $\sum_i\lambda_ie_1\lambda^*_i=e_{\mathcal{P}}$ (see
 \Cref{bases-n-intermediate}). So, by
 \Cref{basis-equivalence}, we obtain
 $\sum_{i,j}\mu_j\lambda_ie_1\lambda^*_i\mu^*_j=\sum_{j}\mu_je_{\mathcal{P}}\mu^*_j=1.$
 Therefore, by \Cref{basis-equivalence} again, $\{\mu_j\lambda_i\}$ is
 a basis for $\mathcal{M}/\mathcal{N}$. On the other hand, we have
\[ \sum_{i,j} \lambda^*_i\mu^*_je_1\mu_j\lambda_i =
   \sum_{i,j}\mu^*_j(\lambda^{\prime}_{i,j})^{*}e_1\lambda^{\prime}_{i,j}\mu_j =\sum_j
   \mu^*_je_{\mathcal{P}}\mu_j  =
   1,
   \]
where the second last equality holds because
$\{\lambda^{\prime}_{i,j}:i\leq i\leq m\}$ is a basis for $
\mathcal{P}/\mathcal{N}$ and the last equality follows because
$\{\mu^*_j:1\leq j\leq n\}$ is a basis for
$\mathcal{M}/\mathcal{P}$. Thus, we conclude that
$\{(\mu_j\lambda_i)^*\}$ is also a basis for
$\mathcal{M}/\mathcal{N}$.  This completes the proof.
\end{proof}

We are now all set to deduce the main theorem of this article.
\begin{theorem}\label{two-sided}
{\em Let $N\subset M$ be a regular subfactor of type $II_1$ with
  finite index. Then, $M$ admits a two-sided Pimsner-Popa basis over $N$.}
\end{theorem}
\begin{proof}
  We observed in \Cref{finitedimension1} that $\mathcal{R}:=N\vee
  \big(N^{\prime}\cap M\big)$ admits a two-sided basis, say,
  $\{\lambda_i\}$, over $N$. Further, we readily deduce, from
  \Cref{p=ep}, that $M$ also admits a two-sided basis, say,
  $\{\mu_j\}$, over $\mathcal{R}$, which is contained in
  $\mathcal{N}_M(\mathcal{N})$. By \Cref{samenormalizer}, we know that
  $\mathcal{N}_M(\mathcal{N}) \subseteq
  \mathcal{N}_M(\mathcal{R})$. Hence, by
  \Cref{two-sided-basis-normalizer}, we conclude that $\{\mu_j \lambda_i
  \}$ is a two-sided Pimsner-Popa basis for $M$ over $N$.
\end{proof}
In view of \Cref{2sided-extremal}, we obtain the following:
\begin{corollary}\label{extremal}
Every regular subfactor of type $II_1$ with finite index is extremal.
  \end{corollary}

It is well known to the experts that every regular subfactor has
integer index; for instance, there is a mention of this fact in
\cite[Page 150]{GHJ} (without a proof). As final application of some
of the results proved above, we deduce this fact along with a precise
expression for the index of such a subfactor. We will use Watatani's
notion of index of a conditional expectation to do so.

Recall, from
\cite{Wat}, that, given an inclusion $B \subset A$ of unital
$C^*$-algebras, a conditional expectation $E: A \rar B$ is said to
have finite index if there exists a right Pimsner-Popa basis
$\{\lambda_i: 1 \leq i \leq n\}$ for $A$ over $B$ via $E$ and the
Watatani index of $E$ is defined as
\[
\mathrm{Ind}(E) = \sum_{i=1}^n \lambda_i \lambda_i^*,
\]
which is independent of the basis $\{\lambda_i\}$ and is an element of
$\mathcal{Z}(A)$.
\begin{theorem}\label{integer-index}
  Every regular subfactor $N \subset M$ of type $II_1$ with  finite Jones index has integer
  valued index and the index is given by
  \[
[M:N] = |G|\, \mathrm{dim}_{\C}(N'\cap M),
\]
where $G$ denotes the generalized Weyl group of the inclusion $N \subset M$.
\end{theorem}
\begin{proof}
Consider the inclusion $\C \subseteq N'\cap M$. Let $\Lambda$ denote
its inclusion matrix. Let $\{\lambda_i\}\subset N'\cap M$ be a
two-sided basis for $N'\cap M$ over $\C$ with respect to $\tr$ as in
\Cref{finitedimensionalbasis}.  We observed in \Cref{finitedimension1}
that $\{\lambda_i\}$ is a two-sided basis for $\mathcal{R}:=N\vee
\big(N^{\prime}\cap M\big)$ as well over $N$ with respect to
${E_N}_{|_{\mathcal{R}}}$.

  Further, from \Cref{p=ep}, $M$ admits a two-sided basis consisting
  of unitaries, say, $\{\mu_j: 1 \leq j \leq |G|\}$, over
  $\mathcal{R}$, which is contained in
  $\mathcal{N}_M(\mathcal{N})$. As seen in \Cref{two-sided}, $\{\mu_j
  \lambda_i \}$ is a two-sided Pimsner-Popa basis for $M$ over $N$
  with respect to $E_N$. Thus, $\{ \lambda_i^*\mu_j^* \}$ is also
  a basis for $M$ over $N$, and we obtain
    \[
[M:N]  = \sum_{i,j} \lambda_i^* \mu_j^* \mu_j
  \lambda_i = |G| \sum_i \lambda_i^* \lambda_i = |G|  \sum_i\lambda_i \lambda_i^* = 
   |G|\, \mathrm{Ind}(\tr),
  \]
where the second last equality holds because $\{\lambda_i\}$ is a
two-sided basis for $\tr$. In particular, $\mathrm{Ind}(\tr)$ is
scalar-valued. So, if $\Lambda$ denotes the  matrix of the inclusion $\C
\subseteq N'\cap M$ and $\bar{s}=(s_1, \ldots, s_k)$ denotes the trace
vector of $\tr$, then by \cite[Corollary 2.4.3]{Wat}, there exists a
$\beta >0$ such that  $\bar{s}\, \Lambda \Lambda^t = \beta \bar{s}$
and $\mathrm{Ind}(\tr) =\beta$. 

Now, if $[n_1, \ldots, n_k]$ is the dimension vector of $N'\cap M$,
then by Watatani's convention, we have $\Lambda = [n_1, \ldots,
  n_k]^t$. Since $\sum_{i=1}^k s_i n_i = 1$, we obtain
\[
\bar{s}\, \Lambda \Lambda^t = \left(\Big(\sum_{i=1}^k s_i n_i\Big)
n_1, \Big(\sum_{i=1}^k s_i n_i\Big) n_2, \ldots, \Big(\sum_{i=1}^k s_i
n_i\Big) n_k \right) = (n_1, n_2, \ldots, n_k),
\]
which yields $\beta = \frac{n_i}{s_i}$ for all $ 1 \leq i \leq
k$. Thus, if $p_i$ denotes a minimal projection in the $i$-th summand
of $N'\cap M$ and $\tilde{p}_i$ denotes the $i$-th minimal central
projection, then $\tr(p_i) = s_i= n_i/\beta$ for all $1 \leq i \leq
k$; so, $\tr(\tilde{p_i}) = n_i^2/\beta = s_i n_i$ for all $1 \leq
i\leq k$. This gives
\[
1 = \tr(1) = \sum_{i=1}^k \tr(\tilde{p}_i) = \sum_{i=1}^k n_i^2/\beta;
\]
so that $\beta = \sum_{i=1}^k n_i^2 = \mathrm{dim}_{\C}(N'\cap M)$. Hence,
\[
[M:N] = |G|  \mathrm{dim}_{\C}(N'\cap M).
\]
This completes the proof.    \end{proof}


\begin{thebibliography}{89}



\bibitem{BG} {\scshape Bakshi, K.C.; Gupta, V.P.} A note on
  quadrilaterals of $II_1$-factors. {\em Internat. J. Math.}  {\bf 30
  } (2019), no. 12, 1950061, 22pp. 
  
\bibitem{BDLR2017} {\scshape Bakshi, K.C.; Das, S.; Liu, Z.; Ren, Y.}
  Angle between intermediate subfactors and its rigidity. {\em
    Trans. Amer. Math. Soc.} {\bf 371} (2019), no. 8,
  5973--5991. 

\bibitem{Cam} {\scshape Cameron, J.M.} Structure results for
  normalizers of $II_1$-factors. {\em Internat. J. Math} {\bf 22}
  (2011), no. 7, 947--979. 

\bibitem{Cho} {\scshape Choda, M.} A characterization of crossed
  product of factors by discrete outer automorphism groups. {\em
    J. Math. Soc. Japan} {\bf 31} (1979), no. 2,
  257--261. 

\bibitem{GHJ} {\scshape Goodman, F.; Harpe, P.de la; Jones, V.F.R.}
  Coxeter graphs and towers of algebras. MSRI Publications, 14. {\em
    Springer, New York}, 1989. 

  \bibitem{Hong} {\scshape Hong, J.H.} A characterization of crossed
    products without cohomology. {\em J. Korean Math. Soc.} {\bf 32}
    (1995), no. 2, 183--192. 

  \bibitem{Jon} {\scshape Jones, V.F.R.} Index for subfactors. {\em
    Invent. Math.} {\bf 72} (1983), no. 1, 1--25.

\bibitem{JoPo} {\scshape Jones, V.F.R.; Popa, S.} Some properties of
  MASAs in factors. {\em Invariant Subspaces and Other Topics}, pp
  89--102, Operator Theory: Adv. Appl., 6, {\em Birkh\"auser,
    Basel-Boston, Mass.,} 1982.

\bibitem{JS} {\scshape Jones, V.F.R.; Sunder, V.S.} Introduction to
  Subfactors. London Mathematical Society Lecture Note Series, 234,
  {\em Cambridge University Press}, 1997.

 \bibitem{Kall} {\scshape Kallman, R.R.} A generalization of free
   action. {\em Duke Math. J.} {\bf 36} (1969), no. 4, 781--789.

 \bibitem{Kos} {\scshape Kosaki, H.} Characterization of crossed
   product (properly infinite case). {\em Pacific J. Math.} {\bf 137}
   (1989), no. 1, 159--167.

\bibitem{Loi} {\scshape Loi, P.H.} On automorphisms of
  subfactors. {\em J. Funct. Anal.} {\bf 141} (1996), no. 2,
  275--293.

\bibitem{Oc} {\scshape Ocneanu, A.} Quantized groups, string algebras
  and Galois theory for algebras. {\em Operator Algebras and
    Applications, Vol. 2}, pp 119--172, London Math. Soc. Lecture
  Notes Series, 136, {\em Cambridge University Press, Cambridge},
  1988. 
  
\bibitem{PiPo} {\scshape Pimsner, M.; Popa, S.} Entropy and index for
  subfactors. {\em Ann. Sci. Ecole Norm. Sup., Series 4} {\bf 19}
  (1986), no. 1,  57--106.

\bibitem{PiPo2} {\scshape Pimsner, M.; Popa, S.} Finite dimensional
  approximation of algebras and obstructions for the index.  {\em
    J. Funct. Anal.} {\bf 98} (1991), no. 2,  270--291.

\bibitem{popa2} {\scshape Popa, S.} Classification of amenable
  subfactors of typeq II. {\em Acta. Math.} {\bf 172} (1994)
  163--255. 
    
  \bibitem{Popa} {\scshape Popa, S.} Classification of subfactors and their
    endomorphisms. CBMS Regional Conference Series in Mathematics, 86.
    {\em AMS, Providence},  1995. 


       \bibitem{popa3} {\scshape Popa, S.} Asymptotic orthogonalization of
         subalgebras in type II$_1$ factors.  {\em Publ.  RIMS Kyoto University} {\bf 55} (2019), no. 4, 795--809.

         
\bibitem{Sun} {\scshape Sunder, V.S.} A model for AF algebras and a
  representation of the Jones projections. {\em J. Operator Theory}
  {\bf 18} (1987) 289--301.

\bibitem{Wat} {\scshape Watatani, Y.} Index for $C^*$-subalgebras.
  {\em Mem. Amer. Math. Soc.}, Vol. 83, No. 424, 1990.


\end{thebibliography}
\end{document}